\documentclass{amsart}
\usepackage[dvips]{graphicx}
\usepackage{amssymb}

\newtheorem{theorem}{Theorem}

\theoremstyle{remark}
\newtheorem{definition}[theorem]{\sc Definition}
\newtheorem{remark}[theorem]{\sc Remark}
\newtheorem{example}[theorem]{\sc Example}
\newtheorem*{acknowledgment}{\sc Acknowledgment}

\numberwithin{equation}{section}

\begin{document}

\title{Fano generalized Bott manifolds}

\author{Yusuke Suyama}
\address{Department of Mathematics, Graduate School of Science, Osaka University,
1-1 Machikaneyama-cho, Toyonaka, Osaka 560-0043 JAPAN}
\email{y-suyama@cr.math.sci.osaka-u.ac.jp}



\date{\today}


\begin{abstract}
We give a necessary and sufficient condition for a generalized Bott manifold
to be Fano or weak Fano.
As a consequence we characterize Fano Bott manifolds.
\end{abstract}

\maketitle

\section{Introduction}

An $m$-stage {\it generalized Bott tower} is a sequence of complex projective space bundles
\begin{equation*}
B_m\stackrel{\pi_m}{\longrightarrow}
B_{m-1}\stackrel{\pi_{m-1}}{\longrightarrow}\cdots\stackrel{\pi_2}{\longrightarrow}
B_1\stackrel{\pi_1}{\longrightarrow}B_0=\{pt\},
\end{equation*}
where $B_j=\mathbb{P}(\mathcal{L}_j^{(1)} \oplus
\cdots \oplus \mathcal{L}_j^{(n_j)} \oplus \mathcal{O}_{B_{j-1}})$
for line bundles $\mathcal{L}_j^{(1)}, \ldots, \mathcal{L}_j^{(n_j)}$ over $B_{j-1}$
and $\mathbb{P}(\cdot)$ denotes the projectivization.
For each $j=1, \ldots, m$,
we call $B_j$ in the sequence a $j$-stage {\it generalized Bott manifold}.
Generalized Bott towers were introduced
by Choi--Masuda--Suh \cite{CMS2}.
When $n_j=1$ for every $j$,
the sequence is called a {\it Bott tower}
and $B_j$ is called a $j$-stage {\it Bott manifold} \cite{GK}.

It is known that any generalized Bott manifold is a nonsingular projective toric variety.
Chary \cite{C} gave the explicit generators of the Kleiman--Mori cone of Bott manifolds
by using toric geometry.
The topology of generalized Bott manifolds was studied in \cite{CMS1, CMS2, PS}.
Recently, Hwang--Lee--Suh \cite{HLS} computed the Gromov width
of generalized Bott manifolds.

A nonsingular projective variety is said to be {\it Fano} (resp.\ {\it weak Fano})
if its anticanonical divisor is ample (resp.\ nef and big).
In this paper, we give a necessary and sufficient condition for a generalized Bott manifold
to be Fano or weak Fano.
To state our main theorem, we introduce some notation.
An $m$-stage generalized Bott manifold is determined by a collection of integers
\begin{equation*}
(a_{j, l}^{(k)})_{2 \leq j \leq m, 1 \leq k \leq n_j, 1 \leq l \leq j-1},
\end{equation*}
see Section 2 for details.
We define $a_{j, l}=(a_{j, l}^{(1)}, \ldots, a_{j, l}^{(n_j)}) \in \mathbb{Z}^{n_j}$
for $2 \leq j \leq m$ and $1 \leq l \leq j-1$.
For a positive integer $n$ and $x=(x_1, \ldots, x_n) \in \mathbb{Z}^n$, we define
$\mu(x)=\min\{0, x_1, \ldots, x_n\}$ and
$\nu(x)=(x_1+\cdots+x_n)-(n+1)\mu(x)$.
Note that $\mu(x) \leq 0$ and $\nu(x) \geq 0$ for any $x \in \mathbb{Z}^n$.
For $1 \leq p \leq m-1$ and $1 \leq q \leq m-p$,
we define $b_{p, q}$ recursively by $b_{p, 1}=a_{p+1, p}$ and
$b_{p, q}=a_{p+q, p}+\sum_{r=1}^{q-1}\mu(b_{p, r})a_{p+q, p+r}$ for $2 \leq q \leq m-p$.
The following is our main theorem:

\begin{theorem}\label{main}
Let $B_m$ be the $m$-stage generalized Bott manifold
determined by a collection $(a_{j, l}^{(k)})$. Then the following hold:
\begin{enumerate}
\item $B_m$ is Fano if and only if
$\sum_{q=1}^{m-p}\nu(b_{p, q}) \leq n_p$ for any $p=1, \ldots, m-1$.
\item $B_m$ is weak Fano if and only if
$\sum_{q=1}^{m-p}\nu(b_{p, q}) \leq n_p+1$ for any $p=1, \ldots, m-1$.
\end{enumerate}
\end{theorem}

Theorem \ref{main} is proved by computing the degree
of each primitive collection of the associated fan.
In a paper of Chary \cite{C},
a characterization of Fano Bott manifolds was claimed,
but there exist counterexamples to his claim (see Example \ref{counterexample}).
As a consequence of Theorem \ref{main},
we give here a characterization of Fano Bott manifolds
(see Theorem \ref{Bott}).

The structure of the paper is as follows:
In Section 2, we recall the construction of the fan associated to a generalized Bott manifold.
In Section 3, we prove Theorem \ref{main} and give some examples.
In Section 4, we characterize Fano Bott manifolds.

\begin{acknowledgment}
The author wishes to thank Professor Akihiro Higashitani for his invaluable comments.
This work was supported by Grant-in-Aid for JSPS Fellows 18J00022.
\end{acknowledgment}

\section{Generalized Bott manifolds}

An $m$-stage {\it generalized Bott tower} is a sequence of complex projective space bundles
\begin{equation*}
B_m\stackrel{\pi_m}{\longrightarrow}
B_{m-1}\stackrel{\pi_{m-1}}{\longrightarrow}\cdots\stackrel{\pi_2}{\longrightarrow}
B_1\stackrel{\pi_1}{\longrightarrow}B_0=\{pt\},
\end{equation*}
where $B_j=\mathbb{P}(\mathcal{L}_j^{(1)} \oplus
\cdots \oplus \mathcal{L}_j^{(n_j)} \oplus \mathcal{O}_{B_{j-1}})$
for line bundles $\mathcal{L}_j^{(1)}, \ldots, \mathcal{L}_j^{(n_j)}$ over $B_{j-1}$.
We call $B_m$ in the sequence an $m$-stage {\it generalized Bott manifold}.
Since the Picard group of $B_{j-1}$
is isomorphic to $\mathbb{Z}^{j-1}$
for any $j=1, \ldots, m$ (see, for example \cite[Exercise II.7.9]{H}),
each line bundle $\mathcal{L}_j^{(k)}$
corresponds to a $(j-1)$-tuple of integers
$(a_{j, 1}^{(k)}, \ldots, a_{j, j-1}^{(k)}) \in \mathbb{Z}^{j-1}$.
Hence an $m$-stage generalized Bott manifold is determined by the collection of integers
\begin{equation*}
(a_{j, l}^{(k)})_{2 \leq j \leq m, 1 \leq k \leq n_j, 1 \leq l \leq j-1}.
\end{equation*}

We recall the construction of the fan $\Delta$ associated to the generalized Bott manifold $B_m$
determined by the collection $(a_{j, l}^{(k)})$.
We follow the notation used in \cite[Section 2]{HLS}.
Let $n=n_1+\cdots+n_m$
and let $e_1^1, \ldots, e_1^{n_1}, \ldots, e_m^1, \ldots, e_m^{n_m}$
be the standard basis for $\mathbb{Z}^n$.
For $l=1, \ldots, m$, we define
\begin{equation*}
u_l^0=-\sum_{k=1}^{n_l}e_l^k+\sum_{j=l+1}^m\sum_{k=1}^{n_j}a_{j, l}^{(k)}e_j^k
\end{equation*}
and $u_l^k=e_l^k$ for $k=1, \ldots, n_l$.
Then the set $\Delta$ of all $n$-dimensional cones of the form
\begin{equation*}
\sum_{l=1}^m(\mathbb{R}_{\geq0}u_l^0+\cdots+\widehat{\mathbb{R}_{\geq0}u_l^{k_l}}
+\cdots+\mathbb{R}_{\geq0}u_l^{n_l}) \subset \mathbb{R}^n
\end{equation*}
with $0 \leq k_l \leq n_l$ for $1 \leq l \leq m$ and their faces
forms a nonsingular complete fan in $\mathbb{R}^n$,
and $B_m$ is the toric variety associated to $\Delta$.

\begin{example}
\begin{enumerate}
\item Let $m=2, n_1=1, n_2=1$.
Then we have
\begin{equation*}
u_1^0=\left(\begin{array}{c} -1 \\ a_{2, 1}^{(1)} \end{array}\right),\quad
u_1^1=\left(\begin{array}{c} 1 \\ 0 \end{array}\right),\quad
u_2^0=\left(\begin{array}{c} 0 \\ -1 \end{array}\right),\quad
u_2^1=\left(\begin{array}{c} 0 \\ 1 \end{array}\right).\quad
\end{equation*}
The fan associated to the 2-stage Bott manifold $B_2$ consists of the cones
\begin{equation*}
\mathbb{R}_{\geq0}u_1^1+\mathbb{R}_{\geq0}u_2^1,\quad
\mathbb{R}_{\geq0}u_1^1+\mathbb{R}_{\geq0}u_2^0,\quad
\mathbb{R}_{\geq0}u_1^0+\mathbb{R}_{\geq0}u_2^1,\quad
\mathbb{R}_{\geq0}u_1^0+\mathbb{R}_{\geq0}u_2^0
\end{equation*}
and their faces.
Thus $B_2$ is a Hirzebruch surface of degree $a_{2, 1}^{(1)}$.
\item Suppose that $a_{j, l}^{(k)}=0$
for any $2 \leq j \leq m, 1 \leq k \leq n_j, 1 \leq l \leq j-1$.
Then the generalized Bott manifold $B_m$
is the product $\mathbb{P}^{n_1}\times\cdots\times\mathbb{P}^{n_m}$.
\end{enumerate}
\end{example}

\section{Proof of Theorem \ref{main}}

First we recall the notion of primitive collections, see \cite{B91, B99} for details.

\begin{definition}
Let $\Delta$ be a nonsingular complete fan in $\mathbb{R}^n$
and $G(\Delta)$ be the set of all primitive generators of $\Delta$.
\begin{enumerate}
\item A nonempty subset $P$ of $G(\Delta)$ is called a {\it primitive collection} of $\Delta$
if $\sum_{u \in P}\mathbb{R}_{\geq0}u$ is not a cone in $\Delta$
but $\sum_{u \in P \setminus \{v\}}\mathbb{R}_{\geq0}u \in \Delta$
for every $v \in P$.
We denote by $\mathrm{PC}(\Delta)$ the set of all primitive collections of $\Delta$.
\item Let $P=\{u_1, \ldots, u_k\} \in \mathrm{PC}(\Delta)$.
Then there exists a unique cone $\sigma$ in $\Delta$
such that $u_1+\cdots+u_k$ is in the relative interior of $\sigma$.
Let $v_1, \ldots, v_l$ be primitive generators of $\sigma$.
Then $u_1+\cdots+u_k=a_1v_1+\cdots+a_lv_l$
for some positive integers $a_1, \ldots, a_l$.
We call this relation the {\it primitive relation} for $P$.
We define $\mathrm{deg}(P)=k-(a_1+\cdots+a_l)$
and we call it the {\it degree} of $P$.
\end{enumerate}
\end{definition}

\begin{theorem}[{\cite[Proposition 2.3.6]{B99}}]\label{degree}
Let $X(\Delta)$ be an $n$-dimensional nonsingular projective toric variety.
Then the following hold:
\begin{enumerate}
\item $X(\Delta)$ is Fano if and only if
$\mathrm{deg}(P)>0$ for every $P \in \mathrm{PC}(\Delta)$.
\item $X(\Delta)$ is weak Fano if and only if
$\mathrm{deg}(P) \geq 0$ for every $P \in \mathrm{PC}(\Delta)$.
\end{enumerate}
\end{theorem}

Let $B_m$ be the $m$-stage generalized Bott manifold
determined by a collection of integers
$(a_{j, l}^{(k)})_{2 \leq j \leq m, 1 \leq k \leq n_j, 1 \leq l \leq j-1}$.
We will use the notation in Section 2 freely.
For $2 \leq j \leq m$ and $1 \leq l \leq j-1$,
we define $a_{j, l}^{(0)}=0$.
For $1 \leq p \leq m-1$ and $1 \leq q \leq m-p$,
we write $b_{p, q}=(b_{p, q}^{(1)}, \ldots, b_{p, q}^{(n_{p+q})}) \in \mathbb{Z}^{n_{p+q}}$
and define $b_{p, q}^{(0)}=0$.
We choose $i_{p, q} \in \{0, 1, \ldots, n_{p+q}\}$ so that $b_{p, q}^{(i_{p, q})}=\mu(b_{p, q})$.
Note that if $\min\{b_{p, q}^{(1)}, \ldots, b_{p, q}^{(n_{p+q})}\}>0$,
then $i_{p, q}$ must be $0$.

\begin{proof}[Proof of Theorem \ref{main}]
For $p=1, \ldots, m$,
let $P_p=\{u_p^0, \ldots, u_p^{n_p}\}$.
Then we can see that $\mathrm{PC}(\Delta)=\{P_1, \ldots, P_m\}$.
Since $u_m^0+\cdots+u_m^{n_m}=0$,
we have $\mathrm{deg}(P_m)=n_m+1>0$.

We regard $a_{j, l}$ and $b_{p, q}$ as column vectors.
For $p=1, \ldots, m-1$, we have
\begin{align*}
&(u_p^0+\cdots+u_p^{n_p})+\sum_{q=1}^{m-p}\left(b_{p, q}^{(i_{p, q})}u_{p+q}^0
+\sum_{k=1}^{n_{p+q}}(b_{p, q}^{(i_{p, q})}-b_{p, q}^{(k)})u_{p+q}^k\right)\\
=&\left(\begin{array}{c}0_{n_1} \\ \vdots \\ 0_{n_p} \\ a_{p+1, p} \\ a_{p+2, p} \\
\vdots \\ a_{m-1, p} \\ a_{m, p} \end{array}\right)
+\left(\begin{array}{c}0_{n_1} \\ \vdots \\ 0_{n_p} \\ -b_{p, 1} \\
{\scriptstyle b_{p, 1}^{(i_{p, 1})}a_{p+2, p+1}} \\
\vdots \\ {\scriptstyle b_{p, 1}^{(i_{p, 1})}a_{m-1, p+1}} \\
{\scriptstyle b_{p, 1}^{(i_{p, 1})}a_{m, p+1}} \end{array}\right)
+\left(\begin{array}{c}0_{n_1} \\ \vdots \\ 0_{n_p} \\ 0_{n_{p+1}} \\
-b_{p, 2} \\
\vdots \\ {\scriptstyle b_{p, 2}^{(i_{p, 2})}a_{m-1, p+2}} \\
{\scriptstyle b_{p, 2}^{(i_{p, 2})}a_{m, p+2}} \end{array}\right)
+\cdots+\left(\begin{array}{c}0_{n_1} \\ \vdots \\ 0_{n_p} \\ 0_{n_{p+1}} \\ 0_{n_{p+2}} \\
\vdots \\ 0_{n_{m-1}} \\ -b_{p, m-p} \end{array}\right)\\
=&0,
\end{align*}
where $0_{n_l}$ is the column vector in $\mathbb{Z}^{n_l}$ with all zero entries.
Since $b_{p, q}^{(i_{p, q})} \leq 0$
and $b_{p, q}^{(i_{p, q})}-b_{p, q}^{(k)} \leq 0$
for any $1 \leq p \leq m-1, 1 \leq q \leq m-p, 1 \leq k \leq n_{p+q}$,
it follows that the primitive relation for $P_p$ is given by
\begin{equation}\label{rel}
u_p^0+\cdots+u_p^{n_p}
=\sum_{q=1}^{m-p}\left((-b_{p, q}^{(i_{p, q})})u_{p+q}^0
+\sum_{k=1}^{n_{p+q}}(b_{p, q}^{(k)}-b_{p, q}^{(i_{p, q})})u_{p+q}^k\right)
\end{equation}
for any $p=1, \ldots, m-1$.
Hence
\begin{align*}
\mathrm{deg}(P_p)&=(n_p+1)+\sum_{q=1}^{m-p}\left(b_{p, q}^{(i_{p, q})}
+\sum_{k=1}^{n_{p+q}}(b_{p, q}^{(i_{p, q})}-b_{p, q}^{(k)})\right)\\
&=(n_p+1)+\sum_{q=1}^{m-p}(-(b_{p, q}^{(1)}+\cdots+b_{p, q}^{(n_{p+q})})
+(n_{p+q}+1)b_{p, q}^{(i_{p, q})})\\
&=(n_p+1)-\sum_{q=1}^{m-p}((b_{p, q}^{(1)}+\cdots+b_{p, q}^{(n_{p+q})})
-(n_{p+q}+1)\mu(b_{p, q}))\\
&=(n_p+1)-\sum_{q=1}^{m-p}\nu(b_{p, q})
\end{align*}
for any $p=1, \ldots, m-1$.
Therefore by Theorem \ref{degree},
$B_m$ is Fano (resp.\ weak Fano) if and only if
$\sum_{q=1}^{m-p}\nu(b_{p, q}) \leq n_p$
(resp.\ $\sum_{q=1}^{m-p}\nu(b_{p, q}) \leq n_p+1$)
for any $p=1, \ldots, m-1$.
This completes the proof of Theorem \ref{main}.
\end{proof}

\begin{remark}
For $p=1, \ldots, m$, let
\begin{equation*}
\tau_p=\sum_{l=1}^{p-1}\sum_{k=1}^{n_l}\mathbb{R}_{\geq 0}u_l^k
+\sum_{k=1}^{n_p-1}\mathbb{R}_{\geq 0}u_p^k
+\sum_{q=1}^{m-p}\sum_{\substack{0 \leq k \leq n_{p+q},\\k \ne i_{p, q}}}
\mathbb{R}_{\geq 0}u_{p+q}^k \subset \mathbb{R}^n.
\end{equation*}
Then $\tau_1, \ldots, \tau_m$ are $(n-1)$-dimensional cones in $\Delta$.
For any $p=1, \ldots, m$,
the wall relation for $\tau_p$
coincides with the primitive relation for $P_p$
(note that the coefficient of $u_{p+q}^{i_{p, q}}$ in (\ref{rel}) is zero for all $p, q$).
Thus the cone of effective 1-cycles of $B_m$ is generated by
the classes of the torus-invariant curves corresponding to $\tau_1, \ldots, \tau_m$.
\end{remark}

\begin{example}\label{example}
\begin{enumerate}
\item Let $m=4, n_1=3, n_2=n_3=n_4=2$.
We consider the 4-stage generalized Bott manifold $B_4$ determined by
\begin{align*}
&a_{2, 1}=(-1, -1),\\
&a_{3, 1}=(0, 0),\quad
a_{3, 2}=(0, -1),\\
&a_{4, 1}=(0, 2),\quad
a_{4, 2}=(0, 1),\quad
a_{4, 3}=(0, 1).
\end{align*}
Then we have
\begin{align*}
&b_{1, 1}=a_{2, 1}=(-1, -1),\\
&b_{1, 2}=a_{3, 1}+\mu(b_{1, 1})a_{3, 2}
=(0, 0)+(-1)(0, -1)=(0, 1),\\
&b_{1, 3}=a_{4, 1}+\mu(b_{1, 1})a_{4, 2}+\mu(b_{1, 2})a_{4, 3}
=(0, 2)+(-1)(0, 1)+0(0, 1)=(0, 1),\\
&b_{2, 1}=a_{3, 2}=(0, -1),\\
&b_{2, 2}=a_{4, 2}+\mu(b_{2, 1})a_{4, 3}
=(0, 1)+(-1)(0, 1)=(0, 0),\\
&b_{3, 1}=a_{4, 3}=(0, 1).
\end{align*}
Since
\begin{align*}
\nu(b_{1, 1})+\nu(b_{1, 2})+\nu(b_{1, 3})&=1+1+1=3 \leq n_1,\\
\nu(b_{2, 1})+\nu(b_{2, 2})&=2+0=2 \leq n_2,\\
\nu(b_{3, 1})&=1 \leq n_3,
\end{align*}
the generalized Bott manifold $B_4$ is Fano.
\item Let $m=3, n_1=n_2=3, n_3=2$.
We consider the 3-stage generalized Bott manifold $B_3$ determined by
\begin{equation*}
a_{2, 1}=(0, -1, -1),\quad
a_{3, 1}=(-4, -2),\quad
a_{3, 2}=(-2, -1).
\end{equation*}
Then we have
\begin{align*}
&b_{1, 1}=a_{2, 1}=(0, -1, -1),\\
&b_{1, 2}=a_{3, 1}+\mu(b_{1, 1})a_{3, 2}=(-4, -2)+(-1)(-2, -1)=(-2, -1),\\
&b_{2, 1}=a_{3, 2}=(-2, -1).
\end{align*}
Since $\nu(b_{1, 1})+\nu(b_{1, 2})=2+3=5>n_1+1$,
the generalized Bott manifold $B_3$ is not weak Fano.
\item For $m=2$ and positive integers $n_1, n_2$,
a 2-stage generalized Bott manifold $B_2$ is Fano if and only if
$(a_{2, 1}^{(1)}+\cdots+a_{2, 1}^{(n_2)})-(n_2+1)\mu(a_{2, 1}) \leq n_1$.
Since any nonsingular projective toric variety with Picard number two
is a 2-stage generalized Bott manifold (see \cite{K}),
these examples exhaust all nonsingular toric Fano varieties with Picard number two.
\end{enumerate}
\end{example}

\section{Fano Bott manifolds}

Let $(\beta_{ij})$ be an $r \times r$ upper triangular integer matrix
whose diagonal entries are one.
We denote by $e_1^+, \ldots, e_r^+$ the standard basis for $\mathbb{Z}^r$
and we put $e_i^-=-\sum_{j=i}^r\beta_{ij}e_j^+$ for $i=1, \ldots, r$.
We define $\Delta$ to be the fan in $\mathbb{R}^r$ such that
$G(\Delta)=\{e_1^+, \ldots, e_r^+, e_1^-, \ldots, e_r^-\}$
and $\mathrm{PC}(\Delta)=\{\{e_i^+, e_i^-\} \mid 1 \leq i \leq r\}$.
The associated toric variety $X_r$ is an $r$-stage Bott manifold.
In the notation of Section 2,
the Bott manifold $X_r$ coincides with $B_m$ determined by
$m=r, n_1=\cdots=n_r=1, a_{j, l}^{(1)}=-\beta_{lj}$ for any $2 \leq j \leq r, 1 \leq l \leq j-1$.

For $i=1, \ldots, r$, we define
\begin{equation*}
\eta_i^+=\{i<j \leq r \mid \beta_{ij}>0\}, \quad
\eta_i^-=\{i<j \leq r \mid \beta_{ij}<0\}.
\end{equation*}
Chary \cite[Theorem 6.3]{C} claimed that $X_r$ is Fano if and only if
for any $i=1, \ldots, r$, at least one of the following holds:
\begin{enumerate}
\item $|\eta_i^+|=0$ and $|\eta_i^-| \leq 1$.
If $|\eta_i^-|=1$ and $\eta_i^-=\{l\}$, then $\beta_{il}=-1$.
\item $|\eta_i^-|=0$ and $|\eta_i^+| \leq 1$.
If $|\eta_i^+|=1$ and $\eta_i^+=\{m\}$, then $\beta_{im}=1$ and $\beta_{mk}=0$ for all $k>m$.
\end{enumerate}
This condition is sufficient but not necessary.
In fact, the following gives a counterexample to the claim:

\begin{example}\label{counterexample}
Suppose that $r \geq 3$ and $\beta_{ij}=1$ for all $1 \leq i \leq r-1$ and $i+1 \leq j \leq r$.
Then $m=r, n_1=\cdots=n_r=1$
and $a_{j, l}^{(1)}=-1$ for any $2 \leq j \leq r, 1 \leq l \leq j-1$
in the notation of Section 2.
We have $b_{p, 1}^{(1)}=-1$
and $b_{p, 2}^{(1)}=\cdots=b_{p, r-p}^{(1)}=0$ for any $p=1, \ldots, r-1$.
Since $\nu(b_{p, 1})+\cdots+\nu(b_{p, r-p})=1$ for any $p=1, \ldots, r-1$,
Theorem \ref{main} (1) implies that $X_r$ is Fano,
but we have $|\eta_1^+|=r-1 \geq 2$.
\end{example}

From Theorem \ref{main} (1),
we deduce a correct characterization of Fano Bott manifolds.
We use the notation in Section 2.
We denote $a_{j, l}^{(1)}$ and $b_{p, q}^{(1)}$
simply by $a_{j, l}$ and $b_{p, q}$ respectively.
Note that $\mu(x)=\min\{0, x\}$ and $\nu(x)=|x|$ for any $x \in \mathbb{Z}$.

\begin{theorem}\label{Bott}
Let $B_m$ be the $m$-stage Bott manifold
determined by a collection $(a_{j, l})_{2 \leq j \leq m, 1 \leq l \leq j-1}$.
Then $B_m$ is Fano if and only if for any $p=1, \ldots, m-1$,
one of the following conditions holds:
\begin{enumerate}
\item $a_{p+1 ,p}=\cdots=a_{m, p}=0$.
\item There exists an integer $q$ with $1 \leq q \leq m-p$ such that
$a_{p+q, p}=1$ and $a_{p+r, p}=0$ for all $r \ne q$.
\item There exists an integer $q$ with $1 \leq q \leq m-p$ such that
\begin{equation*}
a_{p+r, p}=\left\{\begin{array}{ll}
0 & (1 \leq r \leq q-1), \\
-1 & (r=q), \\
a_{p+r, p+q} & (q+1 \leq r \leq m-p). \end{array}\right.
\end{equation*}
\end{enumerate}
\end{theorem}

\begin{proof}
First we show the necessity. Suppose that $B_m$ is Fano. Let $1 \leq p \leq m-1$.
Then we have $\nu(b_{p, 1})+\cdots+\nu(b_{p, m-p}) \leq 1$ by Theorem \ref{main} (1).
Hence it falls into the following two cases:

{\it Case 1}. Suppose that $\nu(b_{p, 1})=\cdots=\nu(b_{p, m-p})=0$.
It follows that $b_{p, 1}=\cdots=b_{p, m-p}=0$
and thus $\mu(b_{p, 1})=\cdots=\mu(b_{p, m-p})=0$.
For any $1 \leq q \leq m-p$,
we have $0=b_{p, q}=a_{p+q, p}+\sum_{l=1}^{q-1}\mu(b_{p, l})a_{p+q, p+l}=a_{p+q, p}$.
It satisfies the condition (1).

{\it Case 2}. Suppose that there exists an integer $q$ with $1 \leq q \leq m-p$
such that $\nu(b_{p, q})=1$ and $\nu(b_{p, r})=0$ for any $r \ne q$.
It follows that $b_{p, r}=0$ for any $r \ne q$.
We have $0=b_{p, r}=a_{p+r, p}+\sum_{l=1}^{r-1}\mu(b_{p, l})a_{p+r, p+l}=a_{p+r, p}$
for any $1 \leq r \leq q-1$
and $b_{p, q}=a_{p+q, p}+\sum_{l=1}^{q-1}\mu(b_{p, l})a_{p+q, p+l}=a_{p+q, p}$.
Since $\nu(a_{p+q, p})=\nu(b_{p, q})=1$,
we must have $a_{p+q, p}=\pm 1$.

{\it Subcase 2.1}. Suppose $a_{p+q, p}=1$.
Then $\mu(b_{p, q})=\mu(a_{p+q, p})=0$.
For any $q+1 \leq r \leq m-p$,
we have $0=b_{p, r}=a_{p+r, p}+\sum_{l=1}^{r-1}\mu(b_{p, l})a_{p+r, p+l}=a_{p+r, p}$.
It satisfies the condition (2).

{\it Subcase 2.2}. Suppose $a_{p+q, p}=-1$.
Then $\mu(b_{p, q})=\mu(a_{p+q, p})=-1$.
For any $q+1 \leq r \leq m-p$,
we have
\begin{align*}
0&=b_{p, r}=a_{p+r, p}+\sum_{l=1}^{r-1}\mu(b_{p, l})a_{p+r, p+l}
=a_{p+r, p}+\mu(b_{p, q})a_{p+r, p+q}\\
&=a_{p+r, p}-a_{p+r, p+q}.
\end{align*}
Thus $a_{p+r, p}=a_{p+r, p+q}$ for any $q+1 \leq r \leq m-p$.
It satisfies the condition (3).

We show the sufficiency. Let $1 \leq p \leq m-1$.
\begin{enumerate}
\item Suppose $a_{p+1 ,p}=\cdots=a_{m, p}=0$.
Then we have $b_{p, 1}=a_{p+1, p}=0$.
If $2 \leq q \leq m-p$ and $b_{p, 1}=\cdots=b_{p, q-1}=0$, then
$b_{p, q}=a_{p+q, p}+\sum_{l=1}^{q-1}\mu(b_{p, l})a_{p+q, p+l}=0$.
Hence it follows by induction that $b_{p, 1}=\cdots=b_{p, m-p}=0$
and thus $\sum_{q=1}^{m-p}\nu(b_{p, q})=0$.
\item Suppose that there exists an integer $q$ with $1 \leq q \leq m-p$ such that
$a_{p+q, p}=1$ and $a_{p+r, p}=0$ for all $r \ne q$.
An argument similar to (1) shows $b_{p, 1}=\cdots=b_{p, q-1}=0$.
Hence $b_{p, q}=a_{p+q, p}+\sum_{l=1}^{q-1}\mu(b_{p, l})a_{p+q, p+l}=a_{p+q, p}=1$.
Since $\mu(b_{p, q})=0$,
it follows by induction that $b_{p, q+1}=\cdots=b_{p, m-p}=0$.
Thus $\sum_{r=1}^{m-p}\nu(b_{p, r})=\nu(b_{p, q})=1$.
\item Suppose that there exists an integer $q$ with $1 \leq q \leq m-p$ such that
\begin{equation*}
a_{p+r, p}=\left\{\begin{array}{ll}
0 & (1 \leq r \leq q-1), \\
-1 & (r=q), \\
a_{p+r, p+q} & (q+1 \leq r \leq m-p). \end{array}\right.
\end{equation*}
An argument similar to (1) shows $b_{p, 1}=\cdots=b_{p, q-1}=0$.
Hence $b_{p, q}=a_{p+q, p}+\sum_{l=1}^{q-1}\mu(b_{p, l})a_{p+q, p+l}=a_{p+q, p}=-1$.
Since $\mu(b_{p, q})=-1$, we have
\begin{equation*}
b_{p, q+1}=a_{p+q+1, p}+\sum_{l=1}^q\mu(b_{p, l})a_{p+q+1, p+l}
=a_{p+q+1, p}-a_{p+q+1, p+q}=0.
\end{equation*}
Furthermore, if $q+2 \leq r \leq m-p$ and $b_{p, q+1}=\cdots=b_{p, r-1}=0$, then
\begin{equation*}
b_{p, r}=a_{p+r, p}+\sum_{l=1}^{r-1}\mu(b_{p, l})a_{p+r, p+l}
=a_{p+r, p}-a_{p+r, p+q}=0.
\end{equation*}
Hence it follows by induction that $b_{p, q+1}=\cdots=b_{p, m-p}=0$
and thus $\sum_{r=1}^{m-p}\nu(b_{p, r})=\nu(b_{p, q})=1$.
\end{enumerate}
Therefore $B_m$ is Fano by Theorem \ref{main} (1). This completes the proof.
\end{proof}

\begin{example}
We consider the case $m=3, n_1=n_2=n_3=1$,
that is, 3-stage Bott manifolds.
If a Bott manifold $B_3$ is Fano, then $(a_{j, l})$ is one of the following
(the types in Table \ref{3-stage} follow the notation used in the book by Oda \cite[p.\ 91]{O}):
\begin{table}[htbp]
\begin{center}
\begin{tabular}{|r|r|r|r||r|r|r|r||r|r|r|r|}
\hline
$a_{2, 1}$ & $a_{3, 1}$ & $a_{3, 2}$ & type &
$a_{2, 1}$ & $a_{3, 1}$ & $a_{3, 2}$ & type &
$a_{2, 1}$ & $a_{3, 1}$ & $a_{3, 2}$ & type \\
\hline
\hline
$0$ & $0$ & $0$ & (6) & $0$ & $1$ & $-1$ & (8) & $1$ & $0$ & $1$ & (10) \\
\hline
$0$ & $0$ & $1$ & (9) & $0$ & $-1$ & $0$ & (9) & $1$ & $0$ & $-1$ & (10) \\
\hline
$0$ & $0$ & $-1$ & (9) & $0$ & $-1$ & $1$ & (8) & $-1$ & $0$ & $0$ & (9) \\
\hline
$0$ & $1$ & $0$ & (9) & $0$ & $-1$ & $-1$ & (7) & $-1$ & $1$ & $1$ & (10) \\
\hline
$0$ & $1$ & $1$ & (7) & $1$ & $0$ & $0$ & (9) & $-1$ & $-1$ & $-1$ & (10) \\
\hline
\end{tabular}
\caption{Fano 3-stage Bott manifolds.}
\label{3-stage}
\end{center}
\end{table}
\end{example}

\end{document}